\batchmode
\makeatletter
\def\input@path{{/HOME1/users/personal/rubsc/Dropbox/}}
\makeatother
\documentclass[english]{article}
\usepackage[T1]{fontenc}
\usepackage[latin9]{inputenc}
\usepackage{geometry}
\usepackage{color}
\usepackage{babel}
\usepackage{units}
\usepackage{mathrsfs}
\usepackage{enumitem}
\usepackage{amsmath}
\usepackage{amsthm}
\usepackage{amssymb}
\usepackage[numbers]{natbib}
\usepackage[unicode=true,
 bookmarks=true,bookmarksnumbered=false,bookmarksopen=false,
 breaklinks=false,pdfborder={0 0 0},pdfborderstyle={},backref=page,colorlinks=true]
 {hyperref}
\hypersetup{pdftitle={nested},
 pdfauthor={Alois Pichler},
 citecolor=cyan, linkcolor= blue, urlcolor= magenta}

\makeatletter
\theoremstyle{plain}
\newtheorem{thm}{\protect\theoremname}
  \theoremstyle{remark}
  \newtheorem{rem}[thm]{\protect\remarkname}
  \theoremstyle{plain}
  \newtheorem{prop}[thm]{\protect\propositionname}
  \theoremstyle{definition}
  \newtheorem{defn}[thm]{\protect\definitionname}
  \theoremstyle{definition}
  \newtheorem{example}[thm]{\protect\examplename}
  \theoremstyle{plain}
  \newtheorem{cor}[thm]{\protect\corollaryname}

\RequirePackage[l2tabu, orthodox]{nag} 
\usepackage{dsfont}  		
\usepackage{newtxtext,newtxmath}	
\usepackage{microtype}		

\usepackage{doi}

\setlist[enumerate,1]{label=(\roman*)}

\DeclareMathOperator{\E}{{\mathds E}}	
\DeclareMathOperator{\one}{{\mathds 1}} 	
\DeclareMathOperator{\pr}{\mathsf{pr}} 	
\DeclareMathOperator{\AVaR}{\mathsf{AV@R}}	
\DeclareMathOperator{\nAVaR}{\mathsf{nAV@R}}	
\DeclareMathOperator{\nd}{\mathsf{d\kern-.15ex I}}	

\DeclareMathOperator*{\essinf}{ess\,inf}
\DeclareMathOperator*{\esssup}{ess\,sup}

\theoremstyle{definition}
\newtheorem{conv}[thm]{Convention}

\makeatother

  \providecommand{\corollaryname}{Corollary}
  \providecommand{\definitionname}{Definition}
  \providecommand{\examplename}{Example}
  \providecommand{\propositionname}{Proposition}
  \providecommand{\remarkname}{Remark}
\providecommand{\theoremname}{Theorem}

\begin{document}

\title{\textbf{Martingale Characterizations Of Risk-Averse Stochastic Optimization
Problems}}

\author{Alois Pichler\thanks{Contact: \protect\href{mailto:alois.pichler@math.tu-chemnitz.de}{alois.pichler@math.tu-chemnitz.de}}
\and  Ruben Schlotter\thanks{Both authors: Technische Universität Chemnitz, 09126 Chemnitz, Germany}}
\maketitle
\begin{abstract}
This paper addresses risk awareness of stochastic optimization problems.
Nested risk measures appear naturally in this context, as they allow
beneficial reformulations for algorithmic treatments. The reformulations
presented extend usual Hamilton\textendash Jacobi\textendash Bellman
equations in dynamic optimization by involving risk awareness in the
problem formulation. 

Nested risk measures are built on risk measures, which originate by
conditioning on the history of a stochastic process. We derive martingale
properties of these risk measures and use them to prove continuity.
It is demonstrated that stochastic optimization problems, which incorporate
risk awareness via nesting risk measures, are continuous with respect
to the natural distance governing these optimization problems, the
nested distance. 


\textbf{Classification:} 90C15, 60B05, 62P05
\end{abstract}

\section{\label{sec:Introduction}Introduction}

Risk measures have been found useful in various disciplines of applied
mathematics, particularly in mathematical finance and in stochastic
optimization. Many applications involve them in various places to
account for risk. It is hence natural to investigate risk measures
in a multistage or dynamic optimization framework as well. One of
the first occurrences of dynamic risk measures in the literature is~\citet{Riedel2004},
conditional risk measures are discussed in~\citet{Ruszczynski} (consider
also the references therein). 

It seems that there is no general consensus on how to incorporate
risk measures in a more general framework which involves time. One
of the conceptual difficulties arising in a problem setting involving
time is time consistency. In short, the decisions considered optimal
at some stage of time should not be rejected from a later perspective. 

Risk-averse multistage stochastic programs incorporate risk awareness
in multistage decision making. These problems have been considered
in \citet{Ruszczynski2010} and \citet{Dentcheva2017}, while applications
can be found in~\citet{PhilpottMatos,PhilpottMatosFinardi} or \citet{Maggioni2012},
e.g., where stochastic dual dynamic programming methods are addressed,
cf.\ also \citet{Romisch,Leclere2015}. In economics, the spread
between \emph{risk-averse} and \emph{risk-neutral} preferences is
associated with a risk or insurance premium. For this, the prevailing
idea of risk in these papers is the interpretation as insurance on
a rolling horizon basis.

\medskip{}

This paper introduces conditional risk functionals based on the history
of the governing stochastic process. These functionals are nested
to obtain risk functionals accounting for the risk at each stage of
the stochastic process. We elaborate their continuity properties and
for important cases we compare them with simple risk measures spanning
the entire horizon as a whole. 

Building on the idea in~\citet{Pflug2009} we introduce the nested
distance via conditional probabilities. We relate these concepts by
verifying that nested risk functionals are continuous with respect
to the nested distance and provide an explicit expression of the modulus
of continuity. 

Martingales are present in stochastic optimization since its very
beginning, cf.\ \citet{Rockafellar1976}. The approach taken here
to verify the results is based on generalized martingales. They reflect
the evolution of risk over time, as risk measures replace risk-neutral
expectations. It is demonstrated that the nested distance, as well
as nested risk measures, follow martingale characteristics in this
generalized sense. 

It is a consequence that risk-averse multistage stochastic programs
are continuous with respect to the nested distance. The optimal solutions
constitute a stochastic process, which again follows a martingale-like
pattern. We finally give a verification theorem. This is a risk-averse
generalization of Hamilton\textendash Jacobi\textendash Bellman equations,
which are well-known from dynamic optimization. 

\paragraph{Outline of the paper.}

Section~\ref{sec:ConditionalRisk} introduces nested risk functionals
after an introductory discussion (Section~\ref{sec:NestedDistance}).
Section~\ref{sec:CostProcess} addresses the main featurs of the
nested distance which are important and relevant to cover the discussion
on continuity of the multistage stochastic programs in Section~\ref{sec:ContinuityRisk}.
Risk martingales are introduced in Section~\ref{sec:Martingale}.
We conclude with the main result in Section~\ref{sec:Continuity}.

\section{\label{sec:NestedDistance}Notation and preliminaries}

We consider the Polish spaces $\big(\Xi_{t},d_{t}\big)$, $t\in\left\{ 1,\dots,T\right\} $.
We shall associate $t\in\left\{ 1,\dots,T\right\} $ with \emph{stage}
or\emph{ time} advancing in discrete steps from $1$ to $T$, where
$T\in\left\{ 1,2,\dots\right\} $ is the time horizon (terminal time)
or final stage. Each space $\Xi_{t}$, $t=1,\dots T$, contains the
information revealed at time $t$. In what follows it will often be
sufficient to consider the spaces $\Xi_{t}=\mathbb{R}^{m_{t}}$. 

The product $\Xi:=\Xi_{1:T}:=\Xi_{1}\times\dots\times\Xi_{T}$ is
endowed with the metric $d$ and
\begin{equation}
\left(\Xi,\,d\right)\label{eq:2}
\end{equation}
is Polish as well (for example, choose $d(x,y):=\ell_{p}(x,y):=\left(\sum_{t=1}^{T}d_{t}(x_{t},y_{t})^{p}\right)^{\nicefrac{1}{p}}$).
We denote elements $x\in\Xi_{1:T}$ by $x_{1:T}:=x=(x_{1},\dots,x_{T})$
and by $\pr_{t}$ the canonical (i.e., coordinate) projection $\pr_{t}(x_{1:T}):=x_{1:t}$
onto the subspace $\Xi_{1:t}:=\Xi_{1}\times\dots\times\Xi_{t}$. To
allow a compact notation we also introduce the empty tuple $x_{1:0}=()$. 

On the Borel sets $\mathcal{F}_{T}:=\mathscr{B}(\Xi_{1:T})$ we consider
the probability measure 
\[
P\colon\mathcal{F}_{T}\to[0,1].
\]
The probability measures restricted to the sub-sigma algebra $\mathcal{F}_{t}:=\sigma(\pr_{t})$
are the image measures defined by 
\[
P_{t}(A):=P^{\pr_{t}}(A)=P\left(A\times\Xi_{t+1}\times\dots\times\Xi_{T}\right),
\]
where $A\in\mathscr{B}(\Xi_{1:t})$, the Borel sigma algebra on $\Xi_{1:t}$.
The sequence $\mathcal{F}:=\mathcal{F}_{0:T}:=\left(\mathcal{F}_{t}\right)_{t=0}^{T}$
is the canonical (i.e., coordinate) filtration and $(\Xi_{1:T},\mathcal{F}_{0:T},P)$
is a filtered probability space (a.k.a.\ stochastic basis), where
we include the trivial sigma algebra $\mathcal{F}_{0}:=\left\{ \emptyset,\,\Xi_{1:T}\right\} $
for completeness and convenience.

The disintegration theorem (cf.~\citet[III-70]{Dellacherie1978}
or \citet[Section~5.3]{Ambrosi2005}) allows `disintegrating' \label{disintegration}
the probability measure with respect to the coordinates. 
\begin{thm}[Disintegration theorem]
\label{thm:Disintegrate}There is a regular kernel, i.e., a $P_{t}$\nobreakdash-a.s.\ uniquely
defined family of measures $P\left(\cdot|\,x_{1:t}\right)$ so that
\begin{enumerate}[noitemsep, nolistsep]
\item \label{enu:Disintegrate1}$x_{1:t}\mapsto P\left(B\mid x_{1:t}\right)$
is measurable for every $B\in\mathscr{B}(\Xi_{t+1}\times\dots\times\Xi_{T})$
and
\item $P(A\times B)=\int_{A}P\left(B\mid x_{1:t}\right)P_{t}(\mathrm{d}x_{1:t})$,
where $A\in\mathscr{B}(\Xi_{1}\times\dots\times\Xi_{t})$ and $B\in\mathscr{B}(\Xi_{t+1}\times\dots\times\Xi_{T})$.
\end{enumerate}
The conditional probability measures 
\begin{equation}
P_{t+1}(\cdot\mid x_{1:t})\;\text{ on }\;\mathscr{B}(\Xi_{t+1})\label{eq:kernel}
\end{equation}
are called (regular) \emph{kernels} and the substring $x_{1:t}$ is
also called a \emph{fiber}.
\end{thm}

By disintegrating the measures $P_{t}$ and composing their kernels
at subsequent stages we obtain the nested expressions 
\begin{equation}
P_{t}(A_{1}\times\dots\times A_{t})=\int_{A_{1}}\int_{A_{2}}\dots\int_{A_{t}}P_{t}\left(\mathrm{d}x_{t}|\,x_{1:t-1}\right)\dots P_{2}\left(\mathrm{d}x_{2}|\,x_{1:1}\right)P_{1}(\mathrm{d}x_{1})\label{eq:4}
\end{equation}
and the conditional probability measures 
\begin{equation}
P(A_{t+1}\times\dots\times A_{T}\mid x_{1:t})=\int_{A_{t+1}}\dots\int_{A_{T}}P_{T}\left(\mathrm{d}x_{T}|\,x_{1:T-1}\right)\dots P_{t+1}\left(\mathrm{d}x_{t+1}|\,x_{1:t}\right).\label{eq:5}
\end{equation}
Both expressions reveal the initial probability measure $P$, which
can be seen by substituting $t=T$ in~\eqref{eq:4} or $t=0$ in~\eqref{eq:5}.
\begin{rem}
The kernels derived from the projected measures~\eqref{eq:kernel}
are conditioned on the history $x_{1:t}$ and they do depend explicitly
on the entire history up to $t$. In the Markovian case this dependence
reduces (simplifies) to 
\[
P_{t+1}(\cdot\mid x_{1:t})=P_{t+1}(\cdot\mid x_{t}).
\]

An important algorithm in stochastic optimization is Stochastic Dual
Dynamic Programming (SDDP). In this context the probabilities are
typically assumed to be \emph{stagewise independent,} i.e., 
\[
P_{t+1}(\cdot\mid x_{1:t})=P_{t+1}(\cdot)
\]
(cf.\ \citet{FilipeGoulart2017}). 
\end{rem}

\section{\label{sec:ConditionalRisk}Conditional  and nested risk measures}

To define conditional risk functionals we recall the definition of
\emph{law invariant, coherent risk functionals} $\mathcal{R}\colon L\to\mathbb{R}$
defined on some vector space $L$ of $\mathbb{R}$\nobreakdash-valued
random variables first. They satisfy the following axioms introduced
by~\citet{Artzner1997}.
\begin{enumerate}[label=A\arabic*, noitemsep]
\item \label{enu:Monotonicity}Monotonicity: $\mathcal{R}\left(Y_{0}\right)\le\mathcal{R}\left(Y_{1}\right)$,
provided that $Y_{0}\le Y_{1}$ almost surely;
\item Translation equivariance: $\mathcal{R}\left(Y+c\right)=\mathcal{R}\left(Y\right)+c$
for $c\in\mathbb{R}$;
\item \label{enu:Convexity} Convexity: $\mathcal{R}\big((1-\lambda)Y_{0}+\lambda Y_{1}\big)\le(1-\lambda)\mathcal{R}\left(Y_{0}\right)+\lambda\mathcal{R}\left(Y_{0}\right)$;
\item \label{enu:Homogeneous}Positively homogeneity: $\mathcal{R}\left(\lambda Y\right)=\lambda\mathcal{R}\left(Y\right)$;
\end{enumerate}
\begin{enumerate}[resume, label=A\arabic*, noitemsep]
\item \label{enu:LawInvariance}Law invariance:  $\mathcal{R}(Y)=\mathcal{R}(Y^{\prime})$,
whenever $Y$ and $Y^{\prime}$ have the same law, i.e., $P(Y\le y)=P(Y^{\prime}\le y)$
for all $y\in\mathbb{R}$.
\end{enumerate}
We shall make frequently use of the following proposition, which is
an immediate consequence of the monotonicity axiom~\ref{enu:Monotonicity}.
\begin{prop}
\label{rem:infimum}The essential infimum of a set of random variables
apparently satisfies $\essinf_{\iota^{\prime}\in I}Y_{\iota^{\prime}}\le Y_{\iota}$
for every $\iota\in I$. Hence, by the monotonicity axiom, \ref{enu:Monotonicity},
$\mathcal{R}\left(\essinf_{\iota^{\prime}\in I}Y_{\iota^{\prime}}\right)\le\mathcal{R}\left(Y_{\iota}\right)$
and subsequently 
\[
\mathcal{R}\left(\essinf_{\iota^{\prime}\in I}Y_{\iota^{\prime}}\right)\le\inf_{\iota\in I}\mathcal{R}\left(Y_{\iota}\right).
\]
\end{prop}

The Average Value-at-Risk at level $\alpha\in[0,1)$ defined on $L^{1}(P)$
by 
\begin{equation}
\AVaR_{\alpha}(Y):=\inf_{q\in\mathbb{R}}\left\{ q+\frac{1}{1-\alpha}\E(Y-q)_{+}\right\} \label{eq:15-1}
\end{equation}
is the most prominent coherent risk functional satisfying the axioms~\ref{enu:Monotonicity}\textendash \ref{enu:LawInvariance}
above. The Average Value-at-Risk at risk level $\alpha=0$ is the
expectation, 
\[
\AVaR_{0}(Y)=\E Y
\]
and, for $Y\in L^{\infty}$, the convenient setting
\[
\AVaR_{1}(Y):=\lim_{\alpha\nearrow1}\AVaR_{\alpha}(Y)=\esssup Y
\]
continuously extends the Average Value-at-Risk to $\alpha=1$.

The Average Value-at-Risk turns out to be of central importance, it
can be interpreted as an extreme point in the set of risk functionals
and, similarly to Choquet's representation, every risk functional
is a convex combination of $\AVaR$s. The following general representation
(Kusuoka's representation, cf.~\citet{Kusuoka}) highlights this
relation. The statement is a consequence of the Fenchel\textendash Moreau
theorem in convex analysis (cf.\ \citet[Lemma~4.55]{Follmer2004}
or \citet{RuszczynskiShapiro2009,Shapiro2011,ShapiroAlois}).
\begin{defn}
A function $\sigma\colon[0,1)\to\mathbb{R}$ is  a \emph{distortion
function}, if $\sigma(\cdot)$ is non-decreasing, $\sigma(\cdot)\ge0$
and $\int_{0}^{1}\sigma(u)\mathrm{d}u=1$.
\end{defn}

\begin{prop}[Derived from Kusuoka's representation, cf.\ \citet{PflugPichlerBuch}]
\label{prop:Kusuoka}Every law invariant, coherent risk functional
$\mathcal{R}\colon L\to\mathbb{R}$ has the representation 
\begin{equation}
\mathcal{R}(Y)=\sup_{\sigma\in\mathcal{S}}\mathcal{R}_{\sigma}(Y),\label{eq:Rsigma}
\end{equation}
where $\mathcal{S}$ is an appropriate collection of distortion functions
and 
\begin{equation}
\mathcal{R}_{\sigma}(Y):=\sup\begin{Bmatrix}\E Y\zeta\,\left|\begin{array}{l}
\zeta\ge0,\,\E\zeta=1\text{ and}\\
\AVaR_{\alpha}(\zeta)\le\frac{1}{1-\alpha}\int_{\alpha}^{1}\sigma(u)\mathrm{d}u\text{ for all }\alpha\in(0,1)
\end{array}\right.\end{Bmatrix}.\label{eq:RDual}
\end{equation}
The vector space $L$ can be assumed to be $L=\left\{ Y\colon\mathcal{R}_{\sigma}(\left|Y\right|)<\infty\right\} $
(cf.\ \citet{Pichler2013a}). In applications, as well in what follows
it will be enough to consider the Lebesgue spaces $L^{p}$, $p\ge1$,
or $L^{\infty}$. 
\end{prop}

The representation of the distortion risk functional~\eqref{eq:Rsigma}
implicitly involves the probability measure $P$ via the expectation
$\E$ and the Average Value-at-Risk in~\eqref{eq:RDual}. We want
to make the probability measure $P$ explicit by rewriting~\eqref{eq:Rsigma}
as 
\begin{equation}
\mathcal{R}(Y)=\mathcal{R}_{\mathcal{S};P}(Y):=\sup\begin{Bmatrix}\E_{P}Y\zeta\,\left|\begin{array}{l}
\zeta\ge0,\,\E_{P}\zeta=1\text{ and}\\
\AVaR_{\alpha;P}(\zeta)\le\frac{1}{1-\alpha}\int_{\alpha}^{1}\sigma(u)\mathrm{d}u,\,\alpha\in(0,1)\\
\qquad\text{ for some }\sigma(\cdot)\in\mathcal{S}
\end{array}\right.\end{Bmatrix},\label{eq:Risk}
\end{equation}
where the expectation in $\AVaR_{\alpha;P}$ is with respect to the
probability measure $P$ as well, cf.~\eqref{eq:15-1}.
\begin{example}
\label{exa:AVaR}The Kusuoka representation of the Average Value-at-Risk
according to Proposition~\ref{prop:Kusuoka} is given by $\mathcal{S}=\left\{ \sigma_{\alpha}(\cdot)\right\} $,
where the distortion function is $\sigma_{\alpha}(u):=\begin{cases}
\frac{1}{1-\alpha} & \text{if }u\ge\alpha,\\
0 & \text{else}.
\end{cases}$
\end{example}

\subsection{Conditional risk measures}

To define conditional versions of risk measures on product spaces
we employ the conditional measures available by the disintegration
theorem, Theorem~\ref{thm:Disintegrate}.
\begin{defn}
Let $\mathcal{S}_{t+1}$ be a collection of distortion functions.
The \emph{conditional risk measure} or risk \emph{measure conditioned
on the fiber $x_{1:t}$ }of the regular kernels of the probability
measure $P$ is
\begin{equation}
\mathcal{R}_{\mathcal{S}_{t+1}}(Y\mid x_{1:t}):=\sup_{\sigma\in\mathcal{S}_{t+1}}\mathcal{R}_{\sigma;P(\cdot\mid x_{1:t})}(Y).\label{eq:16-2}
\end{equation}
\end{defn}

As a consequence of Theorem~\ref{thm:Disintegrate}\ref{enu:Disintegrate1}
and the representations~\eqref{eq:Rsigma} and~\eqref{eq:RDual},
the mapping 
\begin{align}
\mathcal{R}_{\mathcal{S}_{t+1}}(Y\mid\cdot)\colon\Xi_{1:t} & \to\mathbb{R}\nonumber \\
x_{1:t} & \mapsto\mathcal{R}_{\mathcal{S}_{t+1}}(Y\mid x_{1:t})\label{eq:10}
\end{align}
is a random variable on $\Xi_{1:t}$, which is $P_{t}$ a.s.\ well-defined
and measurable with respect to $\mathcal{F}_{t}$. For $t=0$, the
conditional risk functional~\eqref{eq:16-2} is 
\[
\mathcal{R}_{\mathcal{S}_{1}}(Y\mid x_{1:0})=\mathcal{R}_{\mathcal{S}_{1}}(Y)=\sup_{\sigma\in\mathcal{S}_{1}}\mathcal{R}_{\sigma;P}(Y)=\mathcal{R}_{\mathcal{S}_{1}}(Y),
\]
a deterministic number. 

\subsection{Nested risk measures}

The conditional risk measures~\eqref{eq:16-2} are well-defined on
a fiber $x_{1:t}$. As each risk functinoal~\eqref{eq:10} is a random
variable, they can be combined and considered in the following recursive,
or nested way.
\begin{defn}[Nested risk functional]
\label{def:NestedRisk}Let $s,\,t\in\left\{ 1,\dots,T\right\} $
with $s<t$. The \emph{nested risk functional} for a sequence $\mathcal{S}_{s+1:t}:=\mathcal{S}_{s+1}\times\dots\times\mathcal{S}_{t}$
of collections of distortion functionals is 
\begin{equation}
\mathcal{R}_{\mathcal{S}_{s+1:t}}(Y\mid x_{1:s}):=\mathcal{R}_{\mathcal{S}_{s+1}}\Bigl(\dots\mathcal{R}_{\mathcal{S}_{t-1}}\bigl(\mathcal{R}_{\mathcal{S}_{t}}(Y\mid x_{1:t-1})\mid x_{1:t-2}\bigr)\dots\mid x_{1:s}\bigr)\Bigr).\label{eq:5-1}
\end{equation}
\end{defn}

\begin{rem}
The nested risk functional $\mathcal{R}_{\mathcal{S}_{1:T}}(\cdot)$
maps real-valued random variables $Y\colon\Xi\to\mathbb{R}$ defined
on $\Xi$ to the real line. The nested risk functional satisfies generalizations
of the axioms~\ref{enu:Monotonicity}\textendash \ref{enu:Homogeneous},
but it is \emph{not} law invariant any longer, i.e., \ref{enu:LawInvariance}
is not necessarily satisfied.

The construction employed in \citet{Shapiro2015} to discuss rectangular
sets is similar to nested risk measure given in Definition~\ref{def:NestedRisk}
above. Indeed, they can be recovered by choosing the feasible set
as given in the general representation~\eqref{eq:Risk}. A major
difference is given by the fact that law invariant risk functionals
have the Kusuoka representation~\eqref{eq:Risk}, which is not the
case for more general risk functionals.
\end{rem}

Importantly, the nested risk measures are recursive as specified in
the following proposition.
\begin{prop}
The nested risk functional $\mathcal{R}_{\mathcal{S}_{t+1:T}}$ is
recursive, it holds that 
\begin{equation}
\mathcal{R}_{\mathcal{S}_{t+1:T}}(Y\mid x_{1:t})=\mathcal{R}_{\mathcal{S}_{t+1:s}}\left(\mathcal{R}_{\mathcal{S}_{s+1:T}}(Y\mid x_{1:s})\mid x_{1:t}\right)\label{eq:recursive}
\end{equation}
whenever $0\le t<s<T$.
\end{prop}

\begin{proof}
The assertion is an immediate consequence of the recursion~\eqref{eq:5-1}
in Definition~\ref{def:NestedRisk}. 
\end{proof}
\begin{example}[Conditional expectation]
The risk-neutral special case is given by choosing the simplest distortion
functions $\mathcal{S}_{t+1}=\left\{ \one\right\} $, i.e., the distortions
consisting only of the constant function $\sigma(\cdot)=\one(\cdot)=1$.
In this case the risk functional~\eqref{eq:16-2} is 
\[
\mathcal{R}_{\mathcal{S}_{t+1};P}(Y\mid x_{1:t})=\E\left(Y\mid x_{1:t}\right),
\]
i.e., 
\[
\mathcal{R}_{\mathcal{S}_{t+1};P}(Y\mid\cdot)=\E^{\mid\mathcal{F}_{t}}(Y)
\]
(recall that $\E^{\mid\mathcal{F}_{t}}(Y)$ is indeed an $\mathcal{F}_{t}$
random variable). The recursion~\eqref{eq:recursive} reflects the
tower property of the conditional expectation.
\end{example}

\begin{defn}[Nested Average Value-at-Risk, cf.~\citet{PflugRomisch2007}]
The nested Average Value-at-Risk for $\alpha_{s+1:t}\in[0,1]^{t-s}$
is a composition of $\AVaR$s at risk levels dependent on the state
$t$. More explicitly, we set 
\begin{equation}
\nAVaR_{\alpha_{s+1:t}}(Y\mid x_{1:s}):=\AVaR_{\alpha_{s+1};P(\cdot|x_{1:s})}\Bigl(\dots\,\AVaR_{\alpha_{t-1};P(\cdot|x_{1:t-2})}\bigl(\AVaR_{\alpha_{t};P(\cdot|x_{1:t-1})}(Y)\bigr)\bigr)\Bigr).\label{eq:nAVaR}
\end{equation}
\end{defn}

The nested Average Value-at-Risk can be bounded by the Average Value-at-Risk.
Indeed, it follows from \citet[Proposition~4.2]{ShapiroXin} that
$\nAVaR_{\alpha_{1:T};P}(Y)\le\AVaR_{\alpha}(Y)$ provided that the
risk level $\alpha$ satisfies $\alpha\ge1-(1-\alpha_{1})\dots(1-\alpha_{T})$.

\section{\label{sec:CostProcess}The distance adapted to nested risk measures }

Generalizing the concept of distance from probability spaces to filtered
probability spaces corresponds to generalizing the distance from random
variables to stochastic processes. As a metric for probability measures
we recall the Wasserstein distance first here, which we then generalize
to a metric of stochastic processes.

\subsection{Wasserstein metric}

Consider the Polish space $(\Xi,d)$ and probability measures 
\[
P,\,\tilde{P}\colon\mathcal{F}\to[0,1]
\]
on the Borel sigma algebra $\mathcal{F}:=\mathscr{B}(\Xi)$. 
\begin{defn}[Wasserstein metric]
\label{def:Wasserstein}Let $P$ and $\tilde{P}$ be probability
measures on $\Xi$ and $r\in[1,\infty)$. The \emph{Wasserstein metric
of order $r$} with respect to the cost function $c\colon\Xi\times\Xi\to\mathbb{R}$
 is 
\begin{equation}
w_{r}(P,\tilde{P};\,c):=\inf_{\pi}\left(\E_{\pi}c^{r}\right)^{\nicefrac{1}{r}}=\inf_{\pi}\left(\iint_{\Xi\times\Xi}c(x,y)^{r}\,\pi(\mathrm{d}x,\mathrm{d}y)\right)^{\nicefrac{1}{r}},\label{eq:1}
\end{equation}
where the infimum in~\eqref{eq:1} is among all bivariate probability
measures $\pi\in\mathcal{P}(\Xi\times\Xi)$ with marginals $P$ and
$\tilde{P}$, i.e., 
\begin{align}
\pi(A\times\Xi) & =P(A),\quad A\in\mathscr{B}(\Xi)\quad\text{ and}\label{eq:11}\\
\pi(\Xi\times B) & =\tilde{P}(B),\quad B\in\mathscr{B}(\Xi).\label{eq:12}
\end{align}
For the Wasserstein distance of order $r=1$ we shall also write simply
$w(P,\tilde{P})$. 
\end{defn}

\begin{rem}
\label{rem:Wasserstein}The Wasserstein metric introduced in~\eqref{eq:1}
is based on a cost functions $c(\cdot)$ (cf.\ also \citet{Villani2003}).
This setting slightly generalizes the usual definition, which is based
on the distance function $d$ of the space $(\Xi,d)$ in lieu of $c$.
In what follows, this extension will be essential.
\end{rem}

\subsection{The nested distance}

The Wasserstein metric $w_{r}$ introduced in Definition~\ref{def:Wasserstein}
is of course well defined for measures~$P$ and~$\tilde{P}$ on
the product space $\left(\Xi_{1:T},d\right)$. The nested distance
generalizes the Wasserstein metric by involving the filtration in
addition. The filtration carries the information revealed over time.
The filtration considered here is the coordinate filtration, and for
this we may introduce the nested distance on coordinate basis as well,
i.e., sequentially by defining the process stage by stage. 
\begin{defn}[Cost process, nested distance]
\label{def:Nested}Let $P$ and $\tilde{P}$ be probability measures
on $\Xi_{1:T}$, let $r\in[1,\infty)$ and let $c\colon\Xi_{1:T}\times\Xi_{1:T}\to\mathbb{R}$
be a lower semi-continuous (lsc.) function.
\begin{enumerate}
\item Cost process $c_{t}$ for $t=T$ down to $0$: 
\begin{enumerate}
\item The cost function $c_{T}$ on $\Xi_{1:T}\times\Xi_{1:T}$ at terminal
time $T$ is 
\begin{align*}
c_{T}\big(x_{1:T},y_{1:T} & \big):=c\big(x_{1:T},y_{1:T}\big).
\end{align*}
We shall refer to $c_{T}$ also as the \emph{terminal} cost function.
\item The cost functions $c_{t}$ for $t<T$ are defined in a backwards
recursive way by
\begin{equation}
c_{t-1}\big(x_{1:T},y_{1:T}\big):=w_{r}\left(P_{t}\left(\cdot\mid x_{1:t-1}\right),\,\tilde{P}_{t}\left(\cdot\mid y_{1:t-1}\right);\,c_{t}\right),\quad t=T,\dots,1,\label{eq:3}
\end{equation}
where $w_{r}$ is the Wasserstein metric of order $r$.
\item The \emph{cost-process} is the stochastic process $c=\left(c_{t}\right)_{t=0}^{T}$.
\end{enumerate}
\item The nested distance: let $c=\left(c_{t}\right)_{t=0}^{T}$ be the
cost process with terminal cost 
\begin{equation}
c_{T}(\cdot)=d(\cdot),\label{eq:Dist}
\end{equation}
the distance of the space $\Xi_{1:T}$ (cf.~\eqref{eq:2}). The \emph{nested
distance} of order $r\ge1$ of the measures $P$ and $\tilde{P}$
is 
\begin{equation}
\nd_{r}\bigl(P,\tilde{P}\bigr):=c_{0}.\label{eq:8}
\end{equation}
\end{enumerate}
\end{defn}

\begin{rem}
The function $c_{t}$ is defined for $\left(x_{1:T},y_{1:T}\right)\in\Xi_{1:T}\times\Xi_{1:T}$,
but its definition in~\eqref{eq:3} notably involves only the truncated
states $\left(x_{1:t},y_{1:t}\right)\in\Xi_{1:t}\times\Xi_{1:t}$.
The cost function $c_{t}$ thus is unambiguously defined for $\left(x_{1:t},y_{1:t}\right)$,
irrespective of future realization $\left(x_{t+1:T},y_{t+1:T}\right)$.
It follows that $c_{t}$ is $\mathcal{F}_{t}\otimes\mathcal{F}_{t}$
measurable and the cost process $(c_{t})_{t=0}^{T}$ is adapted to
the filtration $\mathcal{F}\otimes\mathcal{F}$.

In particular, $c_{0}$ is independent of the formal argument $\left(x_{1:T},y_{1:T}\right)$
(the string $x_{1:0}$ is empty for $t=0$ in~\eqref{eq:3}) so that
$c_{0}$ is a number ($c_{0}=\nd_{r}\bigl(P,\tilde{P}\bigr)\in\mathbb{R}$)
and the nested distance is well-defined by~\eqref{eq:8}. 
\end{rem}

\begin{rem}
It is a consequence of Hölder's inequality that $w_{r}\bigl(P,\tilde{P}\bigr)\le w_{r^{\prime}}\bigl(P,\tilde{P}\bigr)$
whenever $r\le r^{\prime}$. By monotonicity of~\eqref{eq:3} we
thus get that 
\begin{equation}
\nd_{r}\bigl(P,\tilde{P}\bigr)\le\nd_{r^{\prime}}\bigl(P,\tilde{P}\bigr)\qquad(r\le r^{\prime}).\label{eq:5-2}
\end{equation}
\end{rem}

\begin{rem}[Relation to Wasserstein metric]
For $T=1$ we have $\Xi_{1:T}=\Xi_{1}$ and there are no intermediary
stages present. In this case, the nested distance reduces to the usual
Wasserstein metric and it holds that 
\[
\nd_{r}\bigl(P,\tilde{P}\bigr)=w_{r}\bigl(P,\tilde{P};\,d\bigr)\qquad(T=1).
\]
\end{rem}

\begin{rem}
As for the Wasserstein distance we also write $\nd(P,\tilde{P})$
if the order is $r=1$ (cf.\ Remark~\ref{rem:Wasserstein}).
\end{rem}

An important case in practice is the cost functions, where costs occur
sequentially at every stage and total costs are accumulated over time.
The cost process reflects this additive property, as the following
proposition outlines.
\begin{prop}[Additive cost functions]
\label{prop:Additive}Suppose the terminal cost function is of particular
form 
\begin{equation}
c_{T}(x,y)=\ell_{r}(x,y)=\left(\sum_{t=1}^{T}d_{t}(x_{t},y_{t})^{r}\right)^{\nicefrac{1}{r}},\label{eq:lr}
\end{equation}
where $d_{t}$, $t=1,\dots,T$ are functions on $\Xi_{t}\times\Xi_{t}$
(distance functions, e.g.). Then the process 
\begin{equation}
\tilde{c}_{t}:=\left(c_{t}^{r}-\sum_{j=1}^{t-1}d_{j}^{r}\right)^{\nicefrac{1}{r}}\label{eq:Pflug}
\end{equation}
 satisfies the recursive equations 
\begin{equation}
\tilde{c}_{t-1}^{r}=d_{t-1}^{r}+w_{r}\left(P_{t}\left(\cdot\mid x_{1:t-1}\right),\,\tilde{P}_{t}\left(\cdot\mid y_{1:t-1}\right);\,\tilde{c}_{t}\right)^{r}\label{eq:9-1}
\end{equation}
with $\tilde{c}_{T}=d_{T}$.

Further, the nested distance is 
\[
\nd_{r}\bigl(P,\tilde{P}\bigr)=\tilde{c}_{0}.
\]
\end{prop}

\begin{rem}
The recursive equation~\eqref{eq:9-1} is actually the initial attempt
in defining a distance on the nested spaces $\Xi_{t}\times\mathcal{P}(\Xi_{t-1})$
for the particular case $r=1$, where $\mathcal{P}(\Xi_{t-1})$ is
the set of probability measures on $\Xi_{t-1}$. We refer to~\citet{Pflug2009}
for the initial and complete discussion on nested spaces and nested
distances. 
\end{rem}

\begin{proof}
From~\eqref{eq:3} we have that 
\[
c_{t-1}\left(x_{1:T},y_{1:T}\right)^{r}=w_{r}\left(P_{t}\left(\cdot\mid x_{1:t-1}\right),\,\tilde{P}_{t}\left(\cdot\mid y_{1:t-1}\right);\,c_{t}\right)^{r}.
\]
As $d_{j}$ are $\mathcal{F}_{t-1}$\nobreakdash-measurable for for
every $j<t$ it follows further that 
\[
c_{t-1}\left(x_{1:T},y_{1:T}\right)^{r}=\sum_{j=1}^{t-1}d_{j}^{r}+w_{r}\left(P_{t}\left(\cdot\mid x_{1:t-1}\right),\,\tilde{P}_{t}\left(\cdot\mid y_{1:t-1}\right);\,\left(c_{t}^{r}-\sum_{j=1}^{t-1}d_{j}^{r}\right)^{\nicefrac{1}{r}}\right)^{r}
\]
and hence
\begin{align*}
\tilde{c}_{t-1}^{r}\left(x_{1:T},y_{1:T}\right)^{r} & =c_{t-1}\left(x_{1:T},y_{1:T}\right)^{r}-\sum_{j=1}^{t-2}d_{j}^{r}\\
 & =d_{t-1}^{r}+w_{r}\left(P_{t}\left(\cdot\mid x_{1:t-1}\right),\,\tilde{P}_{t}\left(\cdot\mid y_{1:t-1}\right);\,\left(c_{t}^{r}-\sum_{j=1}^{t-1}d_{j}^{r}\right)^{\nicefrac{1}{r}}\right)^{r}\\
 & =d_{t-1}^{r}+w_{r}\left(P_{t}\left(\cdot\mid x_{1:t-1}\right),\,\tilde{P}_{t}\left(\cdot\mid y_{1:t-1}\right);\,\tilde{c}_{t}\right)^{r},
\end{align*}
which is the assertion.
\end{proof}
\begin{rem}
It is evident that the assertion of the previous statement holds as
well in case of cost functions which are nonanticipative and of the
form $c_{T}(x,y)=\left(\sum_{t=1}^{T}d_{t}(x_{1:t},y_{1:t})^{r}\right)^{\nicefrac{1}{r}}$.
\end{rem}

\subsection{Characterization as a martingale\label{subsec:Martingale}}

For the measure $P$ we have given the nested expressions~\eqref{eq:4}
and~\eqref{eq:5} based on kernels explicitly. In the same way one
may glue together the kernels which are optimal in~\eqref{eq:3}
to compute the nested distance and cost process. To this end denote
the optimal kernels on $\Xi_{t}\times\Xi_{t}$ obtained in~\eqref{eq:3}
by $\pi_{t}(\cdot\times\cdot\mid x_{1:t},y_{1:t})$. A well-known
result of \citet{Brenier1987,Brenier1991} (see also \citet{McCann})
asserts that the Wasserstein problem~\eqref{eq:1} attains the infimum
at a unique bivariate measure $\pi$ for the quadratic cost function
$c(x,y)=\left\Vert x-y\right\Vert ^{2}$, if both measures $P$ and
$\tilde{P}$ have finite variance and do not give mass to small sets
(cf.\ \citet[Theorem~2.12]{Villani2003}); the measures $\pi_{t}(\cdot\times\cdot\mid x_{1:t},y_{1:t})$
thus exist.

The global measure governing all kernels then is 
\begin{align}
\pi\left(A\times B\right): & =\iint_{A_{1}\times B_{1}}\Big(\iint_{A_{2}\times B_{2}}\dots\big(\iint_{A_{T}\times B_{T}}\pi_{T}(\mathrm{d}x_{T},\mathrm{d}y_{T}\mid x_{1:T-1},y_{1:T-1})\big)\label{eq:9}\\
 & \qquad\qquad\qquad\qquad\dots\pi_{2}(\mathrm{d}x_{2},\mathrm{d}y_{2}|\,x_{1},y_{1})\Big)\pi_{1}(\mathrm{d}x_{1},\mathrm{d}y_{1}),\nonumber 
\end{align}
where $A=A_{1}\times\dots\times A_{T}$ and $B=B_{1}\times\dots\times B_{T}$.
The measure $\pi$ is a bivariate measure on the entire space $\Xi_{1:T}\times\Xi_{1:T}$.

We have the following alternative characterization of the governing
bivariate measure~\eqref{eq:9}.
\begin{prop}
The conditional marginals of the measure $\pi$ defined in~\eqref{eq:9}
satisfy 
\begin{align}
\pi\left(A\times\Xi\mid x_{1:t},y_{1:t}\right) & =P\left(A\mid x_{1:t}\right),\qquad A\in\mathcal{F}_{T}\quad\text{ and}\label{eq:44}\\
\pi\left(\Xi\times B\mid x_{1:t},y_{1:t}\right) & =\tilde{P}\left(B\mid y_{1:t}\right),\qquad B\in\mathcal{F}_{T},\label{eq:45}
\end{align}
for every $t\in\left\{ 0,\dots T-1\right\} $. 
\end{prop}

\begin{proof}
The most inner integral in~\eqref{eq:9} satisfies 
\[
\iint_{A_{T}\times\Xi_{T}}\pi(\mathrm{d}x_{T},\mathrm{d}y_{T}|\,x_{1:T-1},y_{1:T-1})=\pi(A_{T}\times\Xi_{T}|\,x_{1:T-1},y_{1:T-1})=P(A_{T}|\,x_{1:T-1})
\]
by construction of the measure $\pi(\cdot,\cdot|\,x_{1:T-1},y_{1:T-1})$.
This is~\eqref{eq:44} for the terminal time $t=T-1$.

Suppose now, by backwards inductions, that the marginal~\eqref{eq:44}
is valid for $t+1$. Then 
\begin{align*}
\pi & \left(A_{t+1:T}\times\Xi_{t+1:T}|\,x_{1:t},y_{1:t}\right)\\
 & =\iint_{A_{t+1}\times\Xi_{t+1}}\dots\iint_{A_{T}\times\Xi_{T}}\pi(\mathrm{d}x_{T},\mathrm{d}y_{T}|\,x_{1:T-1},y_{1:T-1})\dots\pi(\mathrm{d}x_{t+1},\mathrm{d}y_{t+1}|\,x_{1:t},y_{1:t})\\
 & =\iint_{A_{t+1}\times\Xi_{t+1}}\pi(A_{t+2:T}\times\Xi_{t+2:T}|\,x_{1:t+1},y_{1:t+1})\pi(\mathrm{d}x_{t+1},\mathrm{d}y_{t+1}|\,x_{1:t},y_{1:t})\\
 & =\iint_{A_{t+1}\times\Xi_{t+1}}P(A_{t+2:T}|\,x_{1:t+1})\pi(\mathrm{d}x_{t+1},\mathrm{d}y_{t+1}|\,x_{1:t},y_{1:t})\\
 & =\int_{A_{t+1}}P(A_{t+2:T}|\,x_{1:t+1})P(\mathrm{d}x_{t+1}|\,x_{1:t})\\
 & =P(A_{t+1:T}|\,x_{1:t}),
\end{align*}
where we have used the decomposition~\eqref{eq:9}, the induction
hypothesis, the decomposition~\eqref{eq:5} and the setting $A_{t+1:T}:=A_{t+1}\times A_{t+2:T}$,
. We conclude that identity~\eqref{eq:44} is valid for all $t$.

The remaining identity~\eqref{eq:45} follows analogously.
\end{proof}
The process $(c_{t})_{t=0}^{T}$ given in Definition~\ref{def:Nested}
is constructed by recursively averaging with respect to the conditional
measures of $\pi$ given in~\eqref{eq:9}. We thus have the following
characterization as a martingale.
\begin{thm}[Martingale characterization]
\label{thm:Martingale}Let $\pi(\cdot,\cdot)$ be the measure defined
in~\eqref{eq:9} and $r\ge1$. Then the cost process $c=(c_{t}^{r})_{t=1}^{T}$
is a martingale with respect to $\pi$ and the canonical filtration,
i.e., 
\[
c_{t}^{r}=\E_{\pi}\left(c_{t+1}^{r}\mid\mathcal{F}_{t}\otimes\mathcal{F}_{t}\right).
\]
\end{thm}

\begin{proof}
By definition of the process $c_{t}$ in~\eqref{eq:3} we have that
\[
c_{t-1}(x_{1:T},y_{1:T})^{r}=\iint_{\Xi_{t}\times\Xi_{t}}c_{t}(x_{1:T},y_{1:T})^{r}\pi(\mathrm{d}x_{t},\mathrm{d}y_{t}\mid x_{1:t-1},y_{1:t-1}),
\]
where $\pi(\cdot,\cdot\mid x_{1:t-1},y_{1:t-1})$ is the measure with
marginals $P(\cdot\mid x_{1:t-1})$ and $\tilde{P}(\cdot\mid y_{1:t-1})$,
resp.,\ for which the Wasserstein distance attains the infimum in~\eqref{eq:3}.
This is the conditional martingale property for the fibers $(x_{1:t-1},y_{1:t-1})$.
The assertion follows as the measure $\pi$ in~\eqref{eq:9} combines
these optimal, conditional measures.
\end{proof}
\begin{cor}[Alternative characterization]
The nested distance is given by 
\[
\nd_{r}\bigl(P,\tilde{P}\bigr)=\inf_{\pi}\left(\E_{\pi}d^{r}\right)^{\nicefrac{1}{r}}=\inf_{\pi}\left(\iint_{\Xi\times\Xi}d(x,y)^{r}\pi(\mathrm{d}x,\mathrm{d}y)\right)^{\nicefrac{1}{r}},
\]
where the infimum is among all probability measures $\pi\in\mathcal{P}(\Xi\times\Xi)$
satisfying the conditional marginal constraints~\eqref{eq:44}\textendash \eqref{eq:45}.
The infimum is attained for the measure $\pi$ defined in~\eqref{eq:9}.
\end{cor}

\begin{proof}
Let $\pi(\cdot\mid\cdot)$ satisfy the marginals \eqref{eq:44}\textendash \eqref{eq:45}.
Then every conditional measure $\pi(\cdot,\cdot\mid x_{1:t-1},y_{1:t-1})$
satisfies the constraints~\eqref{eq:11}\textendash \eqref{eq:12}
to compute the Wasserstein distance. It follows that $\nd_{r}(P,\tilde{P})^{r}\le\E_{\pi}d^{r}$.

The measure $\pi$ defined in~\eqref{eq:9} satisfies the constraints~\eqref{eq:44}\textendash \eqref{eq:45}
as well. However, we have from Theorem~\ref{thm:Martingale} that
$c_{t}^{r}$ is a martingale. The assertion follows from the power
property of the conditional expectation, as $c_{T}^{r}=d^{r}$ and
\begin{align*}
\nd_{r}(P,\tilde{P})^{r}=c_{0}^{r} & =\E_{\pi}\left(\dots\E_{\pi}\left(c_{t+1}^{r}\mid\mathcal{F}_{t}\otimes\mathcal{F}_{t}\right)\dots\mid\mathcal{F}_{1}\otimes\mathcal{F}_{1}\right)\\
 & =\E_{\pi}\left(\dots\E_{\pi}\left(\dots\E_{\pi}\left(d^{r}\mid\mathcal{F}_{T}\otimes\mathcal{F}_{T}\right)\dots\mid\mathcal{F}_{t}\otimes\mathcal{F}_{t}\right)\dots\mid\mathcal{F}_{1}\otimes\mathcal{F}_{1}\right)\\
 & =\E_{\pi}d^{r};
\end{align*}
hence the result.
\end{proof}
For additive cost functions the distance of the individual stages
have to be taken care of. The following corollary describes the process
in analogy to Proposition~\ref{prop:Additive} above.
\begin{cor}[Additive cost functions]
Let $\pi(\cdot,\cdot)$ be the optimal measure~\eqref{eq:9} and
$c_{T}$ the additive cost function~\eqref{eq:lr} for $r\ge1$.
Then the process
\[
\tilde{c}_{t}^{r}+\sum_{j=1}^{t-1}d_{j}^{r}
\]
is a martingale with respect to the measure $\pi$ (cf.~\eqref{eq:Pflug}).
\end{cor}

\begin{proof}
This is immediate as $\tilde{c}_{t}^{r}=c_{t}^{r}-\sum_{j=1}^{t-1}d_{j}^{r}$
by definition of the process~\eqref{eq:Pflug} and as $c_{t}$ is
a martingale by Theorem~\ref{thm:Martingale}.
\end{proof}

\section{\label{sec:ContinuityRisk}Continuity properties}

The risk functionals defined in~\eqref{eq:Rsigma} above are continuous
with respect to the Wasserstein distance. We generalize the results
here and verify that nested risk functionals are continuous with respect
to the nested distance. This section elaborates the modulus of continuity.
\begin{prop}[Continuity of risk functionals]
\label{prop:Spectal} Let $\mathcal{R}_{\mathcal{S}}$ be a general
risk functional according~\eqref{eq:Risk}. Suppose that the random
variables $Y$, $\tilde{Y}\colon\Xi\to\mathbb{R}$ satisfy
\begin{equation}
Y(x)-\tilde{Y}(y)\le L\cdot d(x,y)^{\beta}\label{eq:Cond}
\end{equation}
for some $\beta\le1$. Then 
\begin{align*}
\mathcal{R}_{\mathcal{S};P}(Y)-\mathcal{R}_{\mathcal{S};\tilde{P}}(\tilde{Y}) & \le L\cdot\sup_{\sigma\in\mathcal{S}}\left\Vert \sigma\right\Vert _{q}\cdot w_{\beta r}(P,\tilde{P})^{\beta}\\
 & \le L\cdot\sup_{\sigma\in\mathcal{S}}\left\Vert \sigma\right\Vert _{q}\cdot w_{r}(P,\tilde{P})^{\beta},
\end{align*}
where $q\in(1,\infty]$ is the Hölder conjugate exponent of $r$ (the
order of the Wasserstein metric) for which $\frac{1}{q}+\frac{1}{r}=1$.
\end{prop}

\begin{proof}
Let $\zeta\ge0$ with $\E\zeta=1$ be chosen so that the supremum
in~\eqref{eq:Risk} is attained up to $\varepsilon>0$, i.e., $\E Y\zeta>\mathcal{R}_{\mathcal{S};P}(Y)-\varepsilon$.
Let $\pi$ have marginals $P$ and $\tilde{P}$. Note that $\E_{\pi}\zeta=\E_{P}\zeta=1$,
so that 
\[
\mathcal{R}_{\mathcal{S};\tilde{P}}(\tilde{Y})=\mathcal{R}_{\mathcal{S};\pi}(\tilde{Y})\ge\E_{\pi}\tilde{Y}\zeta.
\]
It follows from Hölder's inequality that 
\begin{align}
\mathcal{R}_{\mathcal{S};P}(Y)-\varepsilon-\mathcal{R}_{\mathcal{S};\tilde{P}}(\tilde{Y}) & \le\iint_{\Xi\times\Xi}\left(Y(x)-\tilde{Y}(y)\right)\zeta(x)\pi(\mathrm{d}x,\mathrm{d}y)\nonumber \\
 & \le L\iint_{\Xi\times\Xi}d\left(x,y\right)^{\beta}\zeta(x)\pi(\mathrm{d}x,\mathrm{d}y)\nonumber \\
 & \le L\left(\iint_{\Xi\times\Xi}d\left(x,y\right)^{\beta r}\pi(\mathrm{d}x,\mathrm{d}y)\right)^{\nicefrac{1}{r}}\left(\E\zeta^{q}\right)^{\nicefrac{1}{q}}.\label{eq:19}
\end{align}
Now note that $\left(\E\zeta^{q}\right)^{\nicefrac{1}{q}}=\left\Vert \sigma\right\Vert _{q}$
where $\sigma(\cdot):=F_{\zeta}^{-1}(\cdot)\in\mathcal{S}$ is the
generalized inverse distribution function. We obtain the desired result
by taking the infimum in~\eqref{eq:19} over all possible measures
with marginals $P$ and $\tilde{P}$ and after letting $\varepsilon\to0$.

For the remaining inequality observe that 
\[
\left(\E_{\pi}d^{\beta r}\right)^{\nicefrac{1}{\beta r}}=\left\Vert d\right\Vert _{\beta r}\le\left\Vert d\right\Vert _{r}=\left(\E_{\pi}d^{r}\right)^{\nicefrac{1}{r}}
\]
by Hölder's inequality, so that 
\[
\eqref{eq:19}\le L\left(\iint_{\Xi\times\Xi}d\left(x,y\right)^{r}\pi(\mathrm{d}x,\mathrm{d}y)\right)^{\nicefrac{\beta}{r}}\cdot\sup_{\sigma\in\mathcal{S}}\left\Vert \sigma\right\Vert _{q}.
\]
This is the assertion.
\end{proof}
\begin{cor}[Continuity of the Average Value-at-Risk]
\label{prop:AVaR}Suppose that $Y(x)-\tilde{Y}(y)\le L\cdot d(x,y)$.
Then 
\[
\AVaR_{\alpha;P}(Y)-\AVaR_{\alpha;\tilde{P}}(\tilde{Y})\le\frac{L}{1-\alpha}w\left(P,\tilde{P};\,d\right).
\]
\end{cor}

\begin{proof}
This is a special case of Proposition~\ref{prop:Spectal} for $r=1$
and $q=\infty$ (cf.\ Example~\ref{exa:AVaR}).
\end{proof}
\begin{thm}[Continuity of nested risk functionals]
\label{thm:Composition}Suppose that the random variables $Y\colon\Xi\to\mathbb{R}$
is Hölder continuous with constant~$L$ and exponent $\beta\le1$,
\[
\left|Y(x)-Y(y)\right|\le L\cdot d(x,y)^{\beta}.
\]
Then the nested risk functional $\mathcal{R}_{\mathcal{S}_{1:T}}(Y)$
is continuous with respect to the nested distance, it holds that \textup{
\[
\left|\mathcal{R}_{\mathcal{S}_{1:T};P}(Y)-\mathcal{R}_{\mathcal{S}_{1:T};\tilde{P}}(Y)\right|\le\sup_{\sigma\in\mathcal{S}_{t},\,t=1,\dots T}\left\Vert \sigma_{1}\right\Vert _{q}\cdot\dots\left\Vert \sigma_{T}\right\Vert _{q}\cdot L\cdot\nd_{r}\left(P,\tilde{P}\right)^{\beta}.
\]
}
\end{thm}

\begin{proof}
We infer from Proposition~\ref{prop:Spectal} with $\tilde{Y}=Y$
that 
\begin{align}
\mathcal{R}_{\mathcal{S}_{T};P\left(\cdot|\,x_{1:T-1}\right)}(Y)- & \mathcal{R}_{\mathcal{S}_{T};\tilde{P}\left(\cdot|\,y_{1:T-1}\right)}(Y)\nonumber \\
 & \le L\cdot\sup_{\sigma_{T}\in\mathcal{S}_{T}}\left\Vert \sigma_{T}\right\Vert _{q}\cdot w_{r}\left(P\left(\cdot|\,x_{1:T-1}\right),\tilde{P}\left(\cdot|\,y_{1:T-1}\right);\,c_{T}\right)^{\beta},\label{eq:21}
\end{align}
where the terminal cost function is the distance as in the definition
of the nested distance (cf.~\eqref{eq:Dist}),
\[
c_{T}=d.
\]
Define the random variables 
\[
Y_{T-1}(x_{1:T-1}):=\mathcal{R}_{\mathcal{S}_{T};P\left(\cdot|\,x_{1:T-1}\right)}(Y)\quad\text{and \quad}\tilde{Y}_{T-1}(y_{1:T-1}):=\mathcal{R}_{\mathcal{S}_{T};\tilde{P}\left(\cdot|\,y_{1:T-1}\right)}(Y),
\]
so that we have
\[
Y_{T-1}(x_{1:T-1})-\tilde{Y}_{T-1}(y_{1:T-1})\le L\cdot\sup_{\sigma_{T}\in\mathcal{S}_{T}}\left\Vert \sigma_{T}\right\Vert _{q}\cdot c_{T-1}\left(x_{1:T},\,y_{1:T}\right)^{\beta}
\]
by~\eqref{eq:21} and the definition of the process $c_{t}$ in~\eqref{eq:3}.
The random variables $Y_{T-1}$ and $\tilde{Y}_{T-1}$ thus satisfy
the condition~\eqref{eq:Cond} with respect to the cost function
$c_{T-1}$. So we may again apply Proposition~\ref{prop:Spectal}
to the measures $P\left(\cdot\mid x_{1:T-2}\right)$ and $\tilde{P}\left(\cdot\mid y_{1:T-2}\right)$
and repeating this procedure for $t=T-2$ down to $t=0$ gives 
\[
\mathcal{R}_{\mathcal{S}_{1:T};P}(Y)-\mathcal{R}_{\mathcal{S}_{1:T};\tilde{P}}(Y)\le\sup_{\sigma_{t}\in\mathcal{S}_{t},\,t=1,\dots T}\left\Vert \sigma_{1}\right\Vert _{q}\cdot\dots\left\Vert \sigma_{T}\right\Vert _{q}\cdot L\cdot c_{0}^{\beta},
\]
with terminal cost function $c_{T}=d$. We have that $c_{0}=\nd_{r}\left(P,\tilde{P}\right)$
and thus 
\[
\mathcal{R}_{\mathcal{S}_{1:T};P}(Y)-\mathcal{R}_{\mathcal{S}_{1:T};\tilde{P}}(Y)\le\sup_{\sigma_{t}\in\mathcal{S}_{t},\,t=1,\dots T}\left\Vert \sigma_{1}\right\Vert _{q}\cdot\dots\left\Vert \sigma_{T}\right\Vert _{q}\cdot L\cdot\nd_{r}\left(P,\tilde{P}\right)^{\beta}.
\]

The result follows finally by exchanging the probability measures
$P$ and $\tilde{P}$.
\end{proof}
\begin{cor}[Continuity of the nested Average Value-at-Risk]
Suppose that $Y$ is Lipschitz continuous with constant $L$. Then
the nested Average Value-at-Risk, $\nAVaR$, is continuous with respect
to the nested distance $\nd$. More precisely, it holds that 
\[
\left|\nAVaR_{\alpha_{1:T};P}(Y)-\nAVaR_{\alpha_{1:T};\tilde{P}}(Y)\right|\le\frac{L}{1-\alpha}\nd_{r}\left(P,\tilde{P}\right)
\]
for every $r\ge1$, where $\alpha\ge1-(1-\alpha_{1})\cdot\dots(1-\alpha_{T})$
(cf.~\eqref{eq:nAVaR}).
\end{cor}

\begin{proof}
The statement for $r=1$ is immediate by the definition of the nested
Average Value-at-Risk, Corollary~\ref{prop:AVaR} and Theorem~\ref{thm:Composition}.
The statement for general $r\ge1$ follows from~\eqref{eq:5-2}.
\end{proof}

\section{\label{sec:Martingale}Dynamic equations and the martingale property}

In what follows we consider multistage optimization problems with
cost function 
\[
Q\colon\mathcal{Z}_{0:T}\times\Xi_{1:T}\to\mathbb{R},
\]
where a sequence of subsequent decisions $z_{t}\in\mathcal{Z}_{t}$
, $t=0,\dots T,$ is chosen from $\mathcal{Z}_{0:T}=\mathcal{Z}_{0}\times\dots\times\mathcal{Z}_{T}$.
To account for risk-averse decision making under uncertainty we involve
risk functionals at each stage.
\begin{defn}[Policy]
The random variable $z_{t}\colon\Xi\to\mathcal{Z}_{t}$ is a random\emph{
policy} or \emph{decision} at time $t$, $t=0,\dots T$. The decision
$z_{t}$ is \emph{nonanticipative} (or \emph{adapted}) if $z_{t}\colon\Xi\to\mathcal{Z}_{t}$
is $\mathcal{F}_{t}$\nobreakdash-measurable for every $t=0,\dots T$,
abbreviated by $z_{t}\lhd\mathcal{F}_{t}$. The function $z\colon\Xi\to\mathcal{Z}_{1:T}$
with $z(x)_{t}:=z_{t}(x)$ is nonanticipative (adapted; in short,
$z\lhd\mathcal{F}$), if each component $z_{t}$ is nonanticipative
for every $t=0,\dots T$. 
\end{defn}

\begin{rem}
It is a consequence of the Doob\textendash Dynkin lemma that $z_{t}$
is nonanticipative if it depends solely on the information available
at time $t\in\{0,\dots,T\}$, i.e., if $z_{t}(x_{1:T})=\tilde{z}_{t}(x_{1:t})$
for some measurable function $\tilde{z}_{t}\colon\Xi_{1:t}\to\mathcal{Z}_{t}$
(cf.\ \citet[Lemma 1.13]{Kallenberg2002Foundations} or \citet[Theorem~II.4.3]{Shiryaev1996}).
As the filtration $\mathcal{F}=\big(\mathcal{F}_{t}\big)_{t=0}^{T}$
is the coordinate filtration it follows that every nonanticipativative
random decision $z\lhd\mathcal{F}$ can be written explicitly as 
\[
z_{0:T}(x_{1:T})=\begin{pmatrix}\begin{array}{l}
z_{0}\\
z_{1}(x_{1})\\
z_{2}(x_{1},x_{2})\\
\quad\vdots\\
z_{T}(x_{1},\dots,x_{T})
\end{array}\end{pmatrix}
\]
for adequate, measurable functions $z_{t}\colon\Xi_{0:t}\to\mathcal{Z}_{t}$.
\end{rem}

\begin{defn}[Multistage optimization]
\label{def:Multistage}Let $Q\colon\mathcal{Z}_{0:T}\times\Xi_{1:T}\to\mathbb{R}\cup\{\infty\}$
be a lsc.\ cost function. The risk-averse multistage optimization
problem is 
\begin{equation}
\inf_{z_{0:T}\lhd\mathcal{F}_{0:T}}\mathcal{R}_{\mathcal{S}_{1:T}}\Bigl(Q\bigl(z_{0:T}(\cdot);\cdot\bigr)\Bigr),\label{eq:multiStage}
\end{equation}
where the infimum is among all adapted policies $z\lhd\mathcal{F}$.
We emphasize and indicated the random component in~\eqref{eq:multiStage}
by `$\cdot$'.
\end{defn}

\begin{rem}
To avoid confusions or ambiguities regarding the arguments of the
function $Q$ we separate the arguments $z\in\mathcal{Z}_{0:T}$ and
$x\in\Xi_{1:T}$ explicitly and write $Q(z;x)$. This will turn out
helpful in what follows, for example in expressions as $Q(z_{0:t-1},z_{t:T};\,x_{1:t},x_{t+1:T})$.
\end{rem}

\begin{rem}
Constraints of the form $z_{0:t}(x_{1:t})\in\mathscr{Z}_{t}(x_{1:t})\subseteq\mathcal{Z}_{t}$
for some multifunction $\mathscr{Z}_{t}(\cdot)$ appear naturally
in applications involving optimization under uncertainty. They are
easily incorporated in the problem formulation~\eqref{eq:multiStage}
just by employing the function $Q(z_{0:T},x_{1.T})\cdot\one_{\mathscr{Z}_{t}(x_{1:t})}(z_{0:T})$
instead of $Q$. This setting is not advisable for real world implementations,
but convenient for the conceptual treatment envisaged here.
\end{rem}

The multistage problem~\eqref{eq:multiStage} thus consists in finding
optimal functions $z_{0},\,z_{1}(\cdot),\dots,z_{T}(\cdot)$ (only
$z_{0}$ is deterministic) and therefore can be considered as \emph{optimization
on function spaces}. 

\subsection{The essential infimum}

We shall make use of the following interchangeability principle, cf.\ also
\citet{ShapiroInterchangeability}. For $z\in\mathcal{Z}$ fixed,
the mapping $x\mapsto Q(z,x)$ is a random variable for which we write
$Q(z,\cdot)$. In what follows we discuss the expression $\inf_{z}Q(z,\cdot)$
and its measurability. We refer to \citet[Appendix~A]{Karatzas} for
a formal definition of the essential infimum $\essinf_{z\in\mathcal{Z}}Q(z,\cdot)$,
which is a measurable random variable as well. 
\begin{prop}
\label{prop:Essinf}Let $\mathcal{Z}$ be a vector space and consider
all policies with values $z(\cdot)\in\mathcal{Z}$. Then there exists
a sequence $z_{n}(\cdot)$ of simple functions so that
\begin{equation}
\lim_{n\to\infty}Q\big(z_{n}(\cdot),\cdot\big)=\essinf_{z(\cdot)\in\mathcal{Z}}Q(z(\cdot),\cdot)\qquad\text{almost surely}\label{eq:16}
\end{equation}
and $Q\big(z_{n}(\cdot),\cdot\big)$ is nonincreasing.
\end{prop}

\begin{proof}
Denote the set of simple functions $z(\cdot)=\sum_{i=1}^{k}a_{i}\one_{A_{i}}(\cdot)$
by $s$. For $z(\cdot)$ and $z^{\prime}(\cdot)$ simple functions
define 
\begin{equation}
z^{\prime\prime}(x):=\begin{cases}
z(x) & \text{if }Q\big(z(x),x\big)\le Q\big(z^{\prime}(x),x\big),\\
z^{\prime}(x) & \text{else},
\end{cases}\label{eq:Max}
\end{equation}
which is a simple function again and measurable. (The maximization~\eqref{eq:Max}
actually defines a directed set or preorder on $s$.) It holds that
$Q(z^{\prime\prime}(\cdot),\cdot)\le Q(z^{\prime}(\cdot),\cdot)$
and $Q(z^{\prime\prime}(\cdot),\cdot)\le Q(z(\cdot),\cdot)$ and the
set $\left\{ Q(z(\cdot),\cdot)\colon z\in s\right\} $ thus is closed
under pairwise minimization. It follows from \citet[Theorem~A.3]{Karatzas}
that there is a sequence $z_{n}(\cdot)$ of simple functions so that
\[
\essinf_{z(\cdot)}Q\big(z(\cdot),\cdot\big)=\lim_{n\to\infty}Q\big(z_{n}(\cdot),\cdot\big)\qquad\text{almost everywhere}
\]
and thus the assertion.
\end{proof}
\begin{cor}
\label{cor:Inf}Let $s$ be a set of policies containing all simple
functions and suppose that 
\begin{equation}
x\mapsto Q(z,x)\label{eq:usc}
\end{equation}
 is upper semi-continuous for every $z\in\mathcal{Z}$. Then there
exists a sequence $z_{n}(\cdot)$ of policies so that 
\[
\lim_{n\to\infty}Q\big(z_{n}(\cdot),\cdot\big)=\inf_{z\in\mathcal{Z}}Q(z,\cdot)\qquad\text{ almost everywhere}.
\]
\end{cor}

\begin{proof}
The set $s$ contains the constant functions and thus 
\[
\inf_{z\in\mathcal{Z}}Q(z,x)=\inf_{z(\cdot)\in\mathcal{Z}}Q\big(z(x),x\big)\qquad\text{for every }x.
\]
We have that $\left\{ x\colon\inf_{z\in\mathcal{Z}}Q(z,x)<\alpha\right\} =\bigcup_{z\in\mathcal{R}}\left\{ x\colon Q(z,x)<\alpha\right\} $
for every $\alpha\in\mathbb{R}$ so that the additional assumptions
ensure that $x\mapsto\inf_{z\in\mathcal{Z}}Q(z,x)$ is measurable.
The assertion thus follows as 
\[
\inf_{z(\cdot)\in\mathcal{Z}}Q\big(z(\cdot),\cdot\big)=\essinf_{z(\cdot)\in\mathcal{Z}}Q(z(\cdot),\cdot)=\lim_{n\to\infty}Q\big(z_{n}(\cdot),\cdot\big),
\]
where $z_{n}(\cdot)$ is the sequence found in Proposition~\ref{prop:Essinf}.
\end{proof}
\begin{conv}\label{def:Infimum}In what follows we shall always understand
the measurable version when writing $\inf_{z\in\mathcal{Z}}Q(z,\cdot)$,
i.e., we set 
\begin{equation}
\inf_{z\in\mathcal{Z}}Q(z,\cdot):=\essinf_{z(\cdot)\in\mathcal{Z}}Q(z(\cdot),\cdot).\label{eq:Convention}
\end{equation}
\end{conv}

The preceding Corollary~\ref{cor:Inf} provides general conditions
so that the convention is void and automatically valid in these cases. 
\begin{prop}[{Risk functional at the essential infimum, cf.\ \citet[Proposition~6.60]{RuszczynskiShapiro2009}}]
\label{prop:R1}Suppose that $\mathcal{R}$ is continuous at $\inf_{z}Q(z,\cdot)$
with respect to convergence in $L^{p}$. Then it holds that 
\[
\inf_{z(\cdot)\in\mathcal{Z}}\mathcal{R}\big(Q(z(\cdot),\cdot)\big)=\mathcal{R}\left(\inf_{z\in\mathcal{Z}}Q\big(z,\cdot\big)\right).
\]
\end{prop}

\begin{proof}
The result is a consequence Lebesgue's dominated convergence theorem
in view of our setting~\eqref{eq:Convention} and the representation
as nonincreasing limit given in~\eqref{eq:16}. 
\end{proof}

\subsection{Martingale characterization}

Section~\ref{subsec:Martingale}, in particular Theorem~\ref{thm:Martingale},
characterize the nested distance as a martingale process. This concept
extends to the value process of the stochastic optimization problem
when generalizing the concept of martingales. We incorporate risk
awareness in the definition of the martingale term first and characterize
the optimal solution of the multistage stochastic optimization problem
as a martingale with respect to the risk functionals involved.
\begin{defn}[Risk martingale]
The stochastic process $v=(v_{t})_{t=0}^{T}$ is a \emph{submartingale}
(\emph{supermartingale}, resp.) \emph{with respect to the risk functionals}
$\mathcal{R_{S}}_{t}$ (an $\mathcal{R}$\nobreakdash-submartingale,
for short), if 
\begin{equation}
v_{t}\le\mathcal{R_{S}}_{t+1}(v_{t+1})\text{ a.s.}\qquad(v_{t}\ge\mathcal{R_{S}}_{t+1}(v_{t+1})\text{ a.s., resp.})\label{eq:7-1}
\end{equation}
for very $t\in\{0,1,\dots T\}$. The process $v_{t}$ is an $\mathcal{R}$\nobreakdash-martingale,
if \eqref{eq:7-1} holds with equality. 
\end{defn}

For the expectation, $\mathcal{R}=\mathbb{E}$, the notion of an $\mathcal{R}$\nobreakdash-martingale
(sub-, supermartingale, resp.) coincides with the usual term martingale
(sub-, supermartingale, resp.).
\begin{rem}
\label{rem:Martingale}A process $v=(v_{t})_{t=0}^{T}$, which is
an $\mathcal{R}$\nobreakdash-submartingale, satisfies in addition
\[
v_{s}\le\mathcal{R}_{S_{s+1:t}}(v_{t}),\qquad0\le s<t<T.
\]
This follows as the risk functionals $\mathcal{R}_{\mathcal{S}_{t}}$
are monotone (Axiom~\ref{enu:Monotonicity}) and from the recursive
definition of the nested risk functional given in Definition~\ref{def:NestedRisk}.
\end{rem}

\begin{thm}
Let $z=(z_{t})_{t=0}^{T}$ be an adapted policy. Then the process
\begin{equation}
v_{t}(x_{1:t}):=\mathcal{R}_{S_{t+1:T}}\Bigl(Q\big(z_{0:t}(x_{1:t}),z_{t+1:T}(x_{1:t},\cdot);x_{1:t},\cdot\big)\mid x_{1:t}\Bigr)\label{eq:7-2}
\end{equation}
is an $\mathcal{R}$\nobreakdash-martingale with terminal value
\begin{equation}
v_{T}=Q\big(z_{0:T}(\cdot);\cdot\big).\label{eq:MartEnd}
\end{equation}
\end{thm}

\begin{proof}
Choosing $t=T$ in the defining equation~\eqref{eq:7-2} gives $v_{T}(x_{1:T})=Q\big(z_{0:T}(x_{1:T});x_{1:T}\big)$
and thus~\eqref{eq:MartEnd}. 

Apply $\mathcal{R}_{\mathcal{S}_{T}}$ and it follows from~\eqref{eq:MartEnd}
that 
\begin{align*}
\mathcal{R}_{S_{T}}\Bigl(v_{T}\mid x_{1:T-1}\Bigr) & =\mathcal{R}_{S_{T}}\Bigl(Q\big(z_{0:T}(x_{1:T});x_{1:T}\big)\mid x_{1:T-1}\Bigr)\\
 & =\mathcal{R}_{S_{T}}\Bigl(Q\big(z_{0:T-1}(x_{1:T-1}),z_{T:T}(x_{1:T-1},\cdot);x_{1:T-1},\cdot\big)\mid x_{1:T-1}\Bigr)\\
 & =v_{T-1}(x_{1:T-1}),
\end{align*}
as $z$ is adapted. This is the desired martingale property for $t=T-1$.

Apply next $\mathcal{R}_{\mathcal{S}_{T-1}}$ to the latter equation
and observe that
\[
v_{T-2}(x_{1:T-2})=\mathcal{R}_{S_{T-1:T}}\Bigl(Q\big(z_{0:T-1}(x_{1:T-1});x_{1:T-1}\big)\mid x_{1:T-2}\Bigr)=\mathcal{R}_{S_{T-1:T}}\Bigl(v_{T-1}(x_{1:T-1})\mid x_{1:T-2}\Bigr),
\]
which is the assertion for $t=T-1$. The general assertion is immediate
by  repeatedly applying the risk functional corresponding to the individual
stage.
\end{proof}
As a consequence we have the following immediate property of an optimal
policy.
\begin{cor}
\label{cor:Optimal}Let $z_{0:T}^{*}\colon\Xi\to\mathcal{Z}_{0:T}$
be an optimal policy in the multistage stochastic optimization problem~\eqref{eq:multiStage}
and $v^{*}=(v_{t}^{*})_{t=0}^{T}$ the value process~\eqref{eq:7-2}
associated with the policy $z_{0:T}^{*}$. Then $v^{*}$ is an $\mathcal{R}$\nobreakdash-martingale
and the starting value $v_{0}^{*}$ is the solution of the optimization
problem~\eqref{eq:multiStage}.
\end{cor}

\subsection{The value process is a martingale}

Associated with the optimal solution of the reference problem~\eqref{eq:multiStage}
is an optimal policy. We shall characterize the evolution of this
process now by highlighting their martingale properties. 

\begin{defn}[The value process]
Let $z_{0:T}\colon\Xi\to\mathcal{Z}_{0:T}$ be a policy. The \emph{value
process} associated with the policy $z_{0:T}$ is $v(z):=\big(v_{t}(z)\big){}_{t=0}^{T}$.
The marginal functions $v_{t}(z):=v_{t}(z_{0:t-1}\mid\cdot)\colon\Xi_{0:t}\to\mathbb{R}$
are defined by 
\begin{equation}
v_{t}(z_{0:t-1}\mid x_{1:t}):=\inf_{z_{t:T}\lhd\mathcal{F}_{t:T}}\mathcal{R}_{S_{t+1:T}}\Bigl(Q\big(z_{0:t-1}(x_{1:t}),z_{t:T}(x_{1:t},\cdot);x_{1:t},\cdot\big)\mid x_{1:t}\Bigr),\qquad t=0,\dots T,\label{eq:Value}
\end{equation}
where the infimum in~\eqref{eq:Value} is among all adapted processes
$z_{t:T}(x_{1:T})=\begin{pmatrix}\begin{array}{l}
z_{t}(x_{1},\dots x_{t})\\
\qquad\vdots\\
z_{T}(x_{1},\dots,x_{t},\dots,x_{T})
\end{array}\end{pmatrix}$. 
\end{defn}

\begin{rem}
The value process at initial time $t=0$ is 
\[
v_{0}^{*}:=\inf_{z_{0:T}\lhd\mathcal{F}_{0:T}}\mathcal{R}_{\mathcal{S}_{1:T}}\Bigl(Q\bigl(z_{0:T}(\cdot);\cdot\bigr)\Bigr),
\]
this value coincides with the risk-averse multistage stochastic program~\eqref{eq:multiStage}
given in Definition~\ref{def:Multistage}. The quantity~$v_{0}^{*}$
is a deterministic number and not random. 

In addition, we have for $t=T$ that
\[
v_{T}(z_{0:T-1}\mid x_{1:T})=\inf_{z_{T}\lhd\mathcal{F}_{T}}Q\Big(z_{0:T-1}(x_{1:T}),z_{T}(x_{1:T});x_{1:T}\Big),
\]
so that the terminal value function does not involve a risk measure
any longer and the terminal optimization problem is deterministic,
i.e., not random either. 
\end{rem}

\begin{thm}[Submartingale characterization of the value process]
\label{thm:Submartingale}The value process is an $\mathcal{R}$\nobreakdash-submartingale
for any given policy $z_{0:T}$.
\end{thm}

\begin{proof}
We have that 
\begin{align}
\mathcal{R}_{\mathcal{S}_{t+1}}\left(v_{t+1}\right) & =\mathcal{R}_{\mathcal{S}_{t+1}}\left(\inf_{z_{t+1:T}\lhd\mathcal{F}_{t+1:T}}\mathcal{R}_{S_{t+2:T}}\Bigl(Q\big(z_{0:t}(x_{1:t+1}),z_{t+1:T}(x_{1:t+1},\cdot);x_{1:t+1},\cdot\big)\mid x_{1:t+1}\Bigr)\right)\nonumber \\
 & =\inf_{z_{t+1:T}\lhd\mathcal{F}_{t+1:T}}\mathcal{R}_{\mathcal{S}_{t+1}}\left(\mathcal{R}_{S_{t+2:T}}\Bigl(Q\big(z_{0:t}(x_{1:t+1}),z_{t+1:T}(x_{1:t+1},\cdot);x_{1:t+1},\cdot\big)\mid x_{1:t+1}\Bigr)\right)\label{eq:19-1}\\
 & =\inf_{z_{t+1:T}\lhd\mathcal{F}_{t+1:T}}\mathcal{R}_{\mathcal{S}_{t+1:T}}\Bigl(Q\big(z_{0:t}(x_{1:t+1}),z_{t+1:T}(x_{1:t+1},\cdot);x_{1:t+1},\cdot\big)\mid x_{1:t+1}\Bigr),\label{eq:20}
\end{align}
where we have employed~\eqref{eq:Convention} in~\eqref{eq:19-1}. 

The result follows now, as that value process~\eqref{eq:Value} is
the infimum among all $z_{t:T}\lhd\mathcal{F}_{t:T}$, while the infimum
in~\eqref{eq:20} is among $z_{t:T}\lhd\mathcal{F}_{t:T}$, which
is one dimension less.
\end{proof}
Dynamic optimization employs verification theorems which give sufficient
conditions for a solution to the optimal control problem, cf.\ \citet[Theorems~5.1 and 5.2]{Fleming1993}.
The following theorem provides the corresponding statement for the
risk-averse multistage stochastic problem. 
\begin{thm}[Martingale characterization, dynamic equations, verification theorem]
For the value process it holds that 
\[
v_{0}^{*}=\inf_{z_{0:t}\lhd\mathcal{F}_{0:t}}\mathcal{R}_{\mathcal{S}_{1:t}}\bigl(v_{t}(z_{0:t})\bigr),\qquad t\in\{0,1,\dots,T\}.
\]
More generally, for $s<t$ we have the recursive equations  
\begin{equation}
v_{s}(z_{0:s-1})=\inf_{z_{s:t}\lhd\mathcal{F}_{s:t}}\mathcal{R}_{\mathcal{S}_{s+1:t}}\bigl(v_{t}(z_{1:s},z_{s+1:t})\bigr).\label{eq:17}
\end{equation}
\end{thm}

\begin{proof}
Applying the conditional risk functional $\mathcal{R}_{\mathcal{S}_{t}}(\cdot\mid x_{1:t-1})$
to~\eqref{eq:Value} gives
\begin{align}
\mathcal{R}_{\mathcal{S}_{t}}\Big(v_{t}(z_{0:t-1}\mid x_{1:t})| & x_{1:t-1}\Big)=\mathcal{R}_{\mathcal{S}_{t}}\left(\inf_{z_{t:T}\lhd\mathcal{F}_{t:T}}\mathcal{R}_{S_{t+1:T}}\Bigl(Q\big(z_{0:t-1}(x_{1:t}),z_{t:T}(x_{1:t},\cdot);x_{1:t},\cdot\big)\mid x_{1:t}\Bigr)\mid x_{1:t-1}\right)\nonumber \\
 & =\inf_{z_{t:T}\lhd\mathcal{F}_{t:T}}\mathcal{R}_{\mathcal{S}_{t}}\left(\mathcal{R}_{S_{t+1:T}}\Bigl(Q\big(z_{0:t-1}(x_{1:t}),z_{t:T}(x_{1:t},\cdot);x_{1:t},\cdot\big)\mid x_{1:t}\Bigr)\mid x_{1:t-1}\right)\label{eq:18}\\
 & =\inf_{z_{t:T}\lhd\mathcal{F}_{t:T}}\mathcal{R}_{\mathcal{S}_{t:T}}\left(Q\big(z_{0:t-1}(x_{1:t}),z_{t:T}(x_{1:t},\cdot);x_{1:t},\cdot\big)\mid x_{1:t-1}\Bigr)\right),\nonumber 
\end{align}
where we have used the montonicity axiom, \ref{enu:Monotonicity}
and Propositon~\ref{rem:infimum} to obtain ``$\le$'' in~\eqref{eq:18}.
The converse inequality ``$\ge"$ involves the Lebesgue Dominated
Convergence Theorem and is a consequence of Proposition~\ref{prop:R1}.

At this stage take the infimum with respect to $z_{t-1}\lhd\mathcal{F}_{t-1}$
and thus 
\begin{align*}
\inf_{z_{t-1}\lhd\mathcal{F}_{t-1}}\mathcal{R}_{\mathcal{S}_{t}}\left(v_{t}(z_{0:t-1}\mid x_{1:t})|x_{1:t-1}\right) & =\inf_{z_{t-1:T}\lhd\mathcal{F}_{t-1:T}}\mathcal{R}_{\mathcal{S}_{t:T}}\left(Q\big(z_{0:t-1}(x_{1:t}),z_{t:T}(x_{1:t},\cdot);x_{1:t},\cdot\big)\mid x_{1:t-1}\Bigr)\right)\\
 & =v_{t-1}(z_{0:t-2}\mid x_{1:t-1}),
\end{align*}
which is the martingale property of the value process $v(z)$. The
remaining equation~\eqref{eq:17} follows in line with Remark~\ref{rem:Martingale}.

The converse inequalities follow from the submartingale characterization,
Theorem~\ref{thm:Submartingale}. 
\end{proof}
The dynamic equations derived in this section can be employed to characterize
optimal solution of the multistage stochastic optimization problem.
The conceptual advantage lies in the fact that each stage can be considered
for its own. For this the dynamic equations can be employed in algorithms
to improve suboptimal policies at each stage individually. 

\section{\label{sec:Continuity}Continuity of risk-averse multistage programs}

The value of the risk-averse multistage stochastic optimization problem~\eqref{eq:multiStage}
depends on the probability measure $P$. We shall make this explicit
by writing 
\begin{equation}
v_{P}:=\inf_{z_{0:T}\lhd\mathcal{F}_{0:T}}\mathcal{R}_{\mathcal{S}_{1:T};P}\Bigl(Q\bigl(z_{0:T}(\cdot);\cdot\bigr)\Bigr).\label{eq:7-3}
\end{equation}
It is known that the \emph{risk-neutral} version of the multistage
problem~\eqref{eq:7-3} is continuous with respect to changing the
probability measure. 

The following main result elaborates continuity of the \emph{risk-averse}
problem with respect to the nested distance and gives the modulus
of continuity explicitly. 
\begin{thm}[Continuity of the risk-averse MSO problem]
\label{thm:47}Suppose that
\[
x\mapsto Q(z;x),\qquad z\in\mathcal{Z},
\]
 is uniformly Lipschitz, i.e., 
\begin{equation}
\left|Q(z;x)-Q(z;y)\right|\le L\cdot d(x,y)\quad\text{ for all }x,y\in\Xi_{1:T}\text{ and }z\in\mathcal{Z}\label{eq:Lipschitz}
\end{equation}
and
\[
z\mapsto Q(z;x)\qquad(x\in\Xi_{1:T})
\]
is convex for every $x$ fixed. Then the risk-averse optimization
problem~\eqref{eq:multiStage} is continuous with respect to changing
the probability measure. More specifically, we have that 
\[
\left|v_{P}-v_{\tilde{P}}\right|\le\sup_{\sigma\in\mathcal{S}_{t},\,t=1,\dots T}\left\Vert \sigma_{1}\right\Vert _{q}\cdot\dots\left\Vert \sigma_{T}\right\Vert _{q}\cdot L\cdot\nd_{r}(P,\tilde{P}),
\]
where the exponents $r$ and $q$ are Hölder conjugates, $\frac{1}{r}+\frac{1}{q}=1$.
\end{thm}

\begin{rem}
The assumption on Lipschitz continuity of the function $Q$ notably
insures the Convention~\ref{def:Infimum} as $Q$ is particularly
usc., cf.~\eqref{eq:usc}.
\end{rem}

\begin{proof}[Proof of Theorem~\ref{thm:47}]
To compare with the second problem $v_{\tilde{P}}$ define the new
policy 
\[
\tilde{z}_{t}(y_{1:t}):=\E_{\pi}\bigl(z_{t}(x)\mid\pr_{t}(x,y)=y_{1:t}\bigr),
\]
where $\pr_{t}(x_{1:t},y_{1:t}):=y_{1:t}$ is the projection onto
the second marginal and consider the specific random variables 
\begin{equation}
Y_{t}(x_{1:t}):=\mathcal{R}_{\mathcal{S}_{t+1:T};P(\cdot\mid x_{1:t})}\Big(Q\big(z_{0:t},z_{t+1:T}(x_{1:t},\cdot);x_{1:t},\cdot\big)\Big)\label{eq:53}
\end{equation}
 and 
\begin{equation}
\tilde{Y}_{t}(y_{1:t}):=\mathcal{R}_{\mathcal{S}_{t+1:T};\tilde{P}(\cdot\mid y_{1:t})}\Bigl(Q\big(z_{0:t},\tilde{z}_{t+1:T}(y_{1:t},\cdot);y_{1:t},\cdot\big)\Bigr),\label{eq:54}
\end{equation}
where $z_{0:t}\in\mathcal{Z}$ is fixed and `$\cdot$' indicates the
random component. 

For $\varepsilon>0$ pick a policy $z=\bigl(z_{0:t}(x_{1:t})\bigr)_{t=1}^{T}$
so that 
\begin{equation}
v_{P}>\mathcal{R}_{\mathcal{S}_{1:T};P}\bigl(Q(z_{0:T}(\cdot);\cdot)\bigr)-\varepsilon.\label{eq:Gewahlt}
\end{equation}
Further, let the measure $\pi(\cdot,\cdot)$ have conditional marginals
$P(\cdot)$ and $\tilde{P}(\cdot)$ with respect to the nested distance,
cf.~\eqref{eq:9}. 

In line with Definition~\ref{def:Nested} we set $c_{T}:=d$ and
proceed by backwards induction from $t=T$ down to $t=0$. 

\medskip{}

Base case: Note that $Y_{T}(x_{1:T})=Q(z_{0:T};x_{1:T})$ and $\tilde{Y}_{T}(y_{1:T})=Q(z_{0:T};y_{1:T})$.
By Lipschitz continuity~\eqref{eq:Lipschitz} it holds that $\tilde{Y}_{T}(y_{1:T})-Y_{T}(x_{1:T})\le L\cdot d(x_{1:T},y_{1:T})$.
This is the statement 
\begin{equation}
\tilde{Y}_{t}(y_{1:t})-Y_{t}(x_{1:t})\le L\cdot\sup_{\sigma\in\mathcal{S}_{t+1:T}}\left\Vert \sigma_{t+1}\right\Vert _{q}\cdot\ldots\left\Vert \sigma_{T}\right\Vert _{q}\cdot c_{t}(x_{1:t},y_{1:t})\label{eq:17-2}
\end{equation}
for the case $t=T$ (and by setting the empty product to $\prod_{t\in\emptyset}f_{t}:=1$). 

\medskip{}

Inductive step: In what follows we shall employ the statement~\eqref{eq:17-2}
as induction hypothesis and deduce the statement for~$t-1$ instead
of $t$. From Jensen's inequality we infer that 
\begin{equation}
Q\left(\tilde{z}(y);y\right)=Q\left(\E_{\pi}\bigl(z(x)\mid\pr_{t}(x,y)=y\bigr);y\right)\le\E_{\pi}\left(Q\bigl(z(x);y)\mid\pr(x,y)=y\right).\label{eq:7}
\end{equation}
To be more specific we emphasize that $z$ is a vector of functions,
$z=(z_{t})_{t=0}^{T}$ and further, each $z_{t}$ is a function of
the variables $x_{1},\dots,x_{t}$, $z_{t}=z_{t}(x_{1:t})$. Jensen's
inequality applies to each function $z_{t}$ and each argument $x_{t}$
separately, so that the inequality~\eqref{eq:7} is actually the
result of applying Jensen's inequality $t$ times repeatedly at each
stage $t$.

Now let $\zeta$ be chosen so that $\E\tilde{Y}_{t}\zeta>\mathcal{R}_{\mathcal{S}_{t:T};\tilde{P}(\cdot\mid y_{1:t-1})}(\tilde{Y}_{t})-\varepsilon^{\prime}$
and $\AVaR_{\alpha}(\zeta)\le\frac{1}{1-\alpha}\int_{\alpha}^{1}\sigma(u)\mathrm{d}u$
($\alpha\in(0,1)$) for some $\sigma(\cdot)\in\mathcal{S}_{t}$. As
the risk functional is recursive we deduce from~\eqref{eq:53} and~\eqref{eq:54}
that 
\begin{align*}
\tilde{Y}_{t-1}-Y_{t-1}-\varepsilon^{\prime} & =\mathcal{R}_{\mathcal{S}_{t:T};\tilde{P}(\cdot\mid y_{1:t-1})}(\tilde{Y}_{t})-\varepsilon^{\prime}-\mathcal{R}_{\mathcal{S}_{t:T};P(\cdot\mid x_{1:t-1})}(Y_{t})\\
 & \le\E_{\pi}Q\big(\tilde{z}(y);y\big)\zeta(y)-\E_{\pi}Q\big(z(x);x\big)\zeta(y)\\
 & \le\E_{\pi}\E_{\pi}\left(Q\bigl(z(x);y)\mid\pr(x,y)=y\right)\zeta(y)-\E_{\pi}Q\left(z(x);x\right)\zeta(y),
\end{align*}
where we have used~\eqref{eq:7}. By the tower property of the conditional
expectation, Lipschitz continuity~\eqref{eq:Lipschitz} and Hölder's
inequality it follows further that 
\begin{align*}
\tilde{Y}_{t-1}-Y_{t-1}-\varepsilon^{\prime} & \le\E_{\pi}Q\bigl(z(x);y)\zeta(y)-\E_{\pi}Q\left(z(x);x\right)\zeta(y)\\
 & \le\E_{\pi}\zeta(y)c_{t}(x_{1:t},y_{1:t})\\
 & \le L\sup_{\sigma\in\mathcal{S}_{t:T}}\left\Vert \sigma_{t}\right\Vert _{q}\cdot\ldots\left\Vert \sigma_{T}\right\Vert _{q}w_{r}\left(P(\cdot\mid x_{1:t-1}),\tilde{P}(\cdot\mid y_{1:t-1});\,c_{t}\right)\\
 & =L\sup_{\sigma\in\mathcal{S}_{t:T}}\left\Vert \sigma_{t}\right\Vert _{q}\cdot\ldots\left\Vert \sigma_{T}\right\Vert _{q}\cdot c_{t}(x_{1:t-1},y_{1:t-1}),
\end{align*}
as $\pi$ has conditional marginals $\tilde{P}(\cdot\mid y_{1:t-1})$
and $P(\cdot\mid x_{1:t-1})$. By letting $\varepsilon^{\prime}\to0$
we get the assertion~\eqref{eq:17-2} for $t-1$. By repeatedly applying
the previous reasoning we thus get that
\begin{align}
\tilde{Y}_{0}-Y_{0} & \le L\sup_{\sigma\in\mathcal{S}_{1:T}}\left\Vert \sigma_{1}\right\Vert _{q}\cdot\ldots\left\Vert \sigma_{T}\right\Vert _{q}\cdot\nd_{r}(P,\tilde{P}).\label{eq:17-4}
\end{align}
Now note that $v_{P}>Y_{0}-\varepsilon$ by~\eqref{eq:Gewahlt} and
we thus have found a policy $\tilde{z}$ so that $v_{\tilde{P}}\le\tilde{Y}_{0}.$
It follows with~\eqref{eq:17-4} that 
\[
v_{\tilde{P}}-v_{P}\le\tilde{Y}_{0}-\left(Y_{0}-\varepsilon\right)\le L\sup_{\sigma\in\mathcal{S}_{1:T}}\left\Vert \sigma_{1}\right\Vert _{q}\cdot\ldots\left\Vert \sigma_{T}\right\Vert _{q}\cdot\nd_{r}(P,\tilde{P})+\varepsilon.
\]

The result finally follows by letting $\varepsilon\to0$ and by interchanging
the role of $P$ and $\tilde{P}$.
\end{proof}

\section{Summary }

This paper addresses risk-averse stochastic optimization problems.
To define the risk functionals based on partial observations we introduce
conditional risk measures first. They are defined on fibers and can
be composed to nested risk measures. We demonstrate that these nested
risk measures are continuous and we establish the modulus of continuity.
As a consequence, the optimization problems are continuous as well,
these problems inherit the modulus of continuity from the risk functionals.

All results come along with characterizations as generalized martingales.
It is demonstrated that the underlying distance is a usual martingale
with respect to the natural filtration. The value functions are shown
to follow a generalized, risk-averse martingale pattern as well. 

\section{Acknowledgment}

We would like to thank Prof.\ Shapiro for proposing to elaborate
the continuity relations of nested risk measures with respect to the
nested distance.

\bibliographystyle{abbrvnat}
\bibliography{0_HOME1_users_personal_rubsc_Dropbox_LiteraturAlois2}

\end{document}